\documentclass{article}
\usepackage{amsmath,amsfonts,amssymb,graphics,verbatim,amsthm}
\usepackage[dvips]{graphicx}

\title{Poisson--Dirichlet Limit Theorems in Combinatorial Applications via Multi-Intensities}
\author{Richard Arratia  \\ Fred Kochman}
\date{November 2012}

\def\e{\mathbb{E} \,}

\newcommand{\ignore}[1]{}

\newtheorem{theorem}{Theorem}
\newtheorem{proposition}{Proposition}
\newtheorem{lemma}{Lemma}

\begin{document}

\maketitle

\begin{abstract}
We present new, exceptionally efficient
proofs of Poisson--Dirichlet limit theorems for the scaled sizes of irreducible components of 
random elements in the classic combinatorial contexts of arbitrary
assemblies, multisets, and selections, when the components generating functions
satisfy certain standard hypotheses. The proofs exploit a new criterion for
Poisson--Dirichlet limits, originally designed for rapid proofs of Billingsley's
theorem on the scaled sizes of log prime factors of random integers (and some 
new generalizations).  

 Unexpectedly, the technique applies in the present combinatorial 
setting as well, giving, perhaps, a long
sought-after unifying point of view.  The proofs depend also 
on formulas of Arratia and Tavar{\'e} for the mixed moments of counts of components of various 
sizes, as well as formulas of Flajolet and Soria for the asymptotics of generating function 
coefficients.
\end{abstract}

\section{Introduction}

\subsection{Summary}
 The goal of this paper is to provide new and exceptionally
efficient proofs of very general Poisson--Dirichlet limit theorems for
the scaled sizes of components of random elements in the classic
combinatorial contexts of assemblies, multisets, and selections, when
the components generating functions satisfy certain standard
hypotheses. The proofs depend on a fairly new characterization of
convergence in distribution to a Poisson--Dirichlet process,
originally designed to yield a rapid proof of Billingsley's 1972
theorem on the asymptotic scaled sizes of log prime factors of random
integers. That work, including new generalizations of Billingsley's
result to factorizations in wide classes of normed arithmetic
semigroups, was presented in \cite{AKM}.

 Poisson--Dirichlet limit theorems are also available for the
asymptotic scaled sizes of irreducible components of various random
combinatorial objects. The earliest such result was that of Kingman
\cite{King3} and Vershik and Schmidt \cite{VS}, 
applying to irreducible cycles of random permutations
distributed uniformly or, more generally, according to the Ewens
sampling formula.  Analogous results for other random combinatorial
objects were eventually discovered, and by the early 1990's quite
general theorems applying uniformly to members of quite general
families were known. The first such, due to Jennie Hansen \cite{JH},
exploited generating function structure. Subsequent versions, due to
Arratia, Barbour and Tavar{\'e}~\cite{ABT}, invoked significantly weaker
hypotheses and used much different techniques in combinatorial
stochastic processes. Further generalizations continue to be
published.

 Unexpectedly, the techniques of \cite{AKM} were found to apply to the combinatorial regime 
as well, in the presence of generating functions, giving new proofs of results going beyond those 
of \cite{JH} though not as general as those of 
\cite{ABT}; but all the classical cases are included, and the new proofs are extremely rapid.
The commonality of technique may be viewed as providing a unifying framework for Billingsley's
theorem and the combinatorial limit theorems, one which has long been sought. 
(See e.g. the unpublished \cite{King2} by J.F.C. Kingman.\footnote{Kingman proposes one possible 
unifying vantage point if natural density is replaced by harmonic density, in Billingsley's
theorem, but he is apparently dissatisfied with this because recovering the original theorem 
then seems to require the intervention of quite nontrivial auxiliary results.})

\subsection{History, Definitions, Notation}

 The origins of these limit theorems lie in the following earlier results.
Let $p_1 \ge p_2 \ge \cdots$ be the prime factors of a random integer chosen uniformly
from $[1..n]$, and let 
$$
L_j:= \log p_j /\log n.
$$
Or, let $l_1 \ge l_2 \ge \cdots$ be the cycle lengths of a uniform random permutation of length $n$,
and let
$$
L_j:= l_j /n.
$$

 In either case we have
$$
\lim_{n\to \infty} \Pr(L_1 \le t) = \rho(1/t)
$$
where $\rho(\cdot)$ is Dickman's $\rho$, the unique continuous function on $[0,\infty]$
satisfying
$$
\rho(t) = 1 \mbox{ for } 0 \le t \le 1
$$
and
\begin{equation}\label{rhorecur}
t\rho(t) = \int_{t-1}^t \rho(u) du \mbox{ for } 1 \le t <\infty,
\end{equation}
as shown for random integers by Dickman \cite{Dick30} in 1930 and for random permutations by 
Goncharov \cite{Gon} in 1944.

 Nowadays these can be viewed as respective corollaries of a pair of 
later results giving the limiting distributions
of the the entire joint processes $L_1,L_2,\dots$, in the two cases. To state these results we
first define the Poisson--Dirichlet distribution:

 Let $U_1,U_2,\dots$ be iid uniform on $[0,1]$, and define a process $G_1,G_2,\dots$, also 
with values lying in $[0,1]$, via
$$
G_1 = 1-U_1, G_2 = U_1(1-U_2), G_3 = U_1U_2(1-U_3),\dots.
$$
Then the Poisson--Dirichlet process (PD for short) $X_1 \ge X_2\ge \cdots$ is the outcome of sorting
$G_1,G_2,\dots$ into non-increasing order, i.e. 
$$
(X_1\ge X_2 \ge \cdots) = {\bf{SORT}}(G_1,G_2,G_3\dots).
$$
It follows at once from the definition of the $G_i$'s that $X_1 + X_2 + \cdots =1$ almost surely,
and it can be shown that for each $k>0$, $X_1,\dots,X_k$ have the marginal distribution with
joint density function  
$$
f_k(x_1,x_2,\dots,x_k) = 
\frac{1}{x_1\cdots x_k} \rho\left(\frac{1-x_1-\cdots-x_k}{x_k}\right)
$$
on $\{1 \ge x_1 \ge x_2 \ge \cdots \ge 0\}\cap\{x_1 + \cdots + x_k \le 1\}$. (In particular, for 
$k=1$ it follows from this, together with \eqref{rhorecur}, that
$\Pr( X_1 \le t) = \rho(1/t)$.) This explicit distribution function provides an alternative
characterization of PD. There are a number of other characterizations, though we will
not need them here. 

 More generally, for any real parameter $\theta >0$ the Poisson--Dirichlet$(\theta)$ process 
(PD($\theta)$) is defined by replacing each $U_i$ with $U_i^{1/\theta}$ in the definition above. 
(So in particular, PD(1) is just PD.)  Then 
the density functions $f_k$ are replaced by 
$$
f_{\theta,k}(x_1,\dots,x_k) = 
\frac{e^{\gamma\theta}\theta^k \Gamma(\theta) x_k^{\theta-1}}{x_1\cdots x_k} 
g_{\theta}\left(\frac{1-x_1-\cdots-x_k}{x_k}\right)
$$
where $g_{\theta}$ is the unique continuous function on $(0,\infty)$ satisfying
$$
g_{\theta}(t) = \frac{e^{-\gamma \theta} t^{\theta-1}}{\Gamma(\theta)}  \mbox{ for } 0 < t \le 1
$$
and
$$
tg_{\theta}(t) = \theta \int^t_{t-1}g_{\theta}(u)du  \mbox{ for } 1 \le t.
$$
Again there are alternative characterizations which might be more convenient in other contexts; 
see \cite{ABT} for full details.

 We can now state the two original Poisson--Dirichlet limit theorems:
\begin{theorem}[Billingsley, 1972]\label{classicBill}
 Let $p_1 \ge p_2 \ge \cdots $ be the prime factors of a uniform random integer $N \in [1..n]$,
and define $L_{jn} = \log p_j / \log n$, where the latter sequence is padded out with zeros. 
Then for each $k>0$, as $n \to \infty$ the joint distribution of $L_{1n},\dots,L_{kn}$
converges  weakly to the initial $k$-dimensional joint PD(1) distribution.
\end{theorem}

\begin{theorem}[Kingman, 1977; Vershik and Schmidt, 1977]\label{classicKing}
 Let $l_1 \ge l_2 \ge \cdots$ be the cycle lengths of a uniform random $n$-long permutation,
and define $L_{jn} = l_j /n$, where the latter sequence is padded out with zeros. Then 
for each $k>0$, as $n \to \infty$ the joint distribution of $L_{1n},\dots,L_{kn}$
converges  weakly to the initial $k$-dimensional joint PD(1) distribution.
\end{theorem}
 (Since, as noted, $F(t) = \rho(1/t)$ is the cumulative distribution function of the initial 
PD variable, the earlier results of Dickman and Goncharev are indeed respective corollaries 
of the two theorems just stated.)

 Billingsley proved his result before PD was a named and studied distribution, deriving his
limiting $k$-dimensional distributions in the form of certain series expansions not obviously
involving Dickman's $\rho$, of which he appears to have been unaware. 
 In turn, neither Kingman, who had
both named and made a study of PD($\theta$) in \cite{King3}, nor Vershik and Schmidt,
makes any mention of Billingsley's
theorem

 At any rate, as remarked in \cite{King2} it was not until 1984 that a 
publication \cite{Ll} noted that the  
two limiting distributions were identical, which immediately raised the question of why 
this should be so.
Over the years, as already mentioned, analogues of the permutation result were proved for other
random decomposable objects, with various unified methods of proof.  Billingsley's theorem,
on the other hand, remained an isolated result in probabilistic number theory until very
recently (see \cite{AKM}); and although 
several different proofs have appeared, the methods have seemed different from those used 
for the combinatorial results, and the coincidence of limiting distributions has been felt to
lack adequate explanation.

 In the next section, however, we will use the recent criterion, from \cite{AKM}, 
for convergence to
PD to give very brief, self-contained proofs of both theorems above by means of a common 
method; and then we will go on to use the companion characterization of convergence to 
the more general
PD($\theta$), for the promised combinatorial applications, in Section~\ref{gencomb}.
 
\section{Characterization via multi-intensities}\label{charsec}
 The following characterizations of convergence to PD and to PD($\theta$) were originally 
presented in \cite{AKM}.  In what follows, all (random) multisets are (almost surely)
at most countable, with only finite multiplicities.

 Given a sequence $A_n$ of random multisubsets of $(0,1]$, let
  $T_n$ denote the sum of the elements of $A_n$, counting
  multiplicities;
and for any multiset $A$ and set $S$ in $(0,1]$ let $|A \cap S|$ denote the cardinality of the 
intersection, also counting multiplicities. 
Also, let $L(n)=(L_1(n),L_2(n),\dots)$, where $L_i(n) := $ the $i^{\rm th}$
largest
element of $A_n$ if $i \le |A_n|$, and   $L_i(n) := 0$ if $i >
|A_n|$. (Our hypotheses will ensure that, almost surely, 
no $A_n$ possesses a positive accumulation point, 
ensuring in turn that the elements can actually be placed in a non-increasing sequence.)

 Here, first, is the PD-only version:
\begin{proposition}\label{PDonly}
 Suppose that
$T_n \le 1$ almost surely, for all $n$, and that for any collection of
disjoint closed intervals $I_i = [a_i,b_i] \subset (0,1], i =
  1,\dots,k$ satisfying $b_1 + \cdots + b_k <1$, for any $k \ge 1$, 
we have 
\begin{equation}\label{simple}
\liminf \e |A_n \cap I_1|\cdots |A_n \cap I_k| \ge 
\prod_{i=1}^k  (\log(b_i) - \log(a_i))
\end{equation}
as $n \to \infty.$
Then $L(n)$ converges in distribution to $(L_1,L_2,\dots)$, the Poisson--Dirichlet distribution
with parameter 1.
\end{proposition}

 For arbitrary PD($\theta$) we also have
\begin{proposition}\label{maintheta}
 Let $\theta >0$.  Suppose that
$T_n \le 1$ almost surely, for all $n$, and that for some 
$ -\infty < \alpha, \beta < \infty$ with $\alpha + \beta
=1 - \theta,$ it is the case that for any collection of disjoint
closed $I_i = [a_i,b_i] \subset (0,1], \ i=1,\dots,k$ satisfying 
$b_1+\cdots + b_k < 1$, for any  $k \ge 1$, 
we have 
\begin{multline}\label{theta intensineq}
 \liminf_{n \to \infty}  \e \prod_{i=1}^k  | A_n \cap I_i|  \ \ge\\
\frac{\theta^k}{(1-a_1-\cdots-a_k)^{\alpha}(1-b_1-\cdots-b_k)^{\beta}}\prod_{i=1}^k (\log(b_i) - \log(a_i)).
\end{multline}
Then
 $L(n)$ converges in distribution to $(L_1,L_2,\dots)_{\theta}$, the Poisson--Dirichlet
process with \mbox{parameter $\theta$.}
\end{proposition}

 Both propositions are proved in \cite{AKM}.  Note that the condition $T_n \le 1$ 
ensures that $A_n$ can possess no positive accumulation
point, as promised.  

 To show at once how rapidly results can be derived using the above criteria, 
here are complete,
self-contained proofs of Theorems \ref{classicBill} and \ref{classicKing}, 
much briefer than any by previous methods. (The present proof of Theorem \ref{classicBill} 
was already presented in \cite{AKM}, but since
it is so brief, and to make the point that there is a single over-arching methodology, 
we reproduce it here.)

 To prove Theorem \ref{classicBill},
let $A_n$ be the multiset whose elements are $\log p/\log n$ for all prime factors $p$ 
of a random $1 \le N \le n$, including 
any multiple copies, and let $A_n^1$ be the underlying set, i.e. with all positive 
multiplicities truncated down to 1. For any prime $p$ at all 
let $I(p|N)$ denote the indicator function
of the event $p|N$. Note that since $\log p_1 + \log p_2 + \cdots = \log N \le \log n$ we 
automatically get $T_n \le 1$. Note also that for any test interval $[a_i,b_i]$ we have
$\log p / \log n \in \{ |A_n \cap [a_i,b_i]|\}$ if and only if $n^{a_i} \le p \le n^{b_i}$. Thus,
writing things out explicitly for $k = 2$, we get 
$$
 E\{ |A_n \cap [a_1,b_1]|\  |A_n \cap [a_2,b_2]|\} 
\ge E\{ |A^1_n \cap [a_1,b_1]|\  |A^1_n \cap [a_2,b_2]|\}
$$
$$
 = E \sum_{n^{a_1}\le p \le n^{b_1}}I(p|N) \sum_{n^{a_2} \le q \le n^{b_2}}I(q|N) 
$$
$$
 = \sum_{n^{a_1} \le p \le n^{b_1}} \sum_{n^{a_2} \le q \le n^{b_2}} E \{I(pq|N)\}
$$
$$
 =\sum_{n^{a_1}\le p \le n^{b_1}} 
\sum_{n^{a_2} \le q \le n^{b_2}}\frac{1}{pq} + O\left(\frac{n^{b_1 + b_2}}{n}\right)
$$
$$
 =(\log b_1 -\log a_1)(\log b_2 -\log a_2) + o(1).
$$
The second equality exploits the fact that always $p \ne q$ since 
they must lie in disjoint intervals, while the third equality depends on
the estimate $\left|\Pr(p|N) -1/p\right| \le 1/n$ together with the fact that there are at most
$n^{b_1}n^{b_2}$ summands. The fourth uses Mertens' formula together with the hypothesis that 
$b_1 + b_2 < 1$. Now take the $\liminf$ as $n \to \infty$. QED.

 To prove Theorem \ref{classicKing}, let $A_n$ be the multiset whose elements are the quotients $l/n$
where $l$ ranges over the lengths of all irreducible cycles of a random permutation of length
$n$. Trivially we have $T_n = n$.  Note that for any test interval $[a_i,b_i]$ we have
$l/n \in \{ |A_n \cap [a_i,b_i]|\}$ if and only if $a_in \le l \le b_in$.  

 Also, for any
positive integer $j$ let $C_j$ be the number of cycles of length $j$ in a random permutation of
length $n$.  Then it is well known that provided $j_1 + \cdots + j_k \le n$ we have 
$E\{C_1\cdots C_k\} = \frac{1}{j_1\cdots j_k}$ exactly. Thus, again writing 
things out explicitly for $k = 2$, we get
$$
 E\{ |A_n \cap [a_1,b_1]|  |A_n \cap [a_2,b_2]|\}
$$
$$
 = E \sum_{a_1n \le j \le b_1n}C_j \sum_{a_2n \le k \le b_2n}C_k 
$$
$$
=  \sum_{a_1n \le j \le b_1n} \sum_{a_2n \le k \le b_2n} E{C_j C_k}
$$
$$
= \sum_{a_1n \le j \le b_1n} \sum_{a_2n \le k \le b_2n} \frac{1}{jk} = 
\left(\sum_{a_1n \le j \le b_1n}\frac{1}{j}\right)\left(\sum_{a_2n \le k \le b_2n} \frac{1}{k}\right)
$$
$$
= (\log b_1 -\log a_1)(\log b_2 -\log a_2) + o(1),
$$
where the requirement $j+k \le n$ is enforced by the hypothesis $b_1 + b_2 <1$, and 
this time we need only know about harmonic sums.
Again take the lim inf of both sides, QED.

\section{PD limit theorems for general combinatorial families}\label{gencomb}

 In this section and the next we present our main results, 
namely, new proofs of generalizations
of the permutation result to randomly selected 
decomposable combinatorial objects. That is, we prove Poisson--Dirichlet
limit theorems for the non-increasing sequence of scaled sizes of irreducible components
of random decomposable objects, as total size grows. 
We cover the three classical families, namely, 
Labeled Assemblies, Multisets, and Selections. The objects are chosen 
equiprobably or, more generally, with their selection probabilities ``tilted'' 
to be proportional to
$$
\phi^K,
$$
where $\phi >0$ is some fixed parameter and $K = K(\mbox{object})$ is the 
total number of irreducible 
components of an object.  (Thus $\phi=1$ corresponds to equiprobable selection.)

\subsection{The Master Theorem}
 The proofs for the three families will be patterned after the proof of 
Theorem \ref{classicKing}, 
as presented in Section~\ref{charsec}.  Specifically, they will all be corollaries of the
following Master Theorem.  We suppose given a family of objects of various weights or sizes
$n$, where $n$ is a positive integer; that there are finitely many objects of each size $n$; 
that each object decomposes somehow into finitely
many irreducible objects, uniquely up to ordering; and that the size of an object is
equal to the sum of the sizes of its irreducible components.  

 Given an object of size $n$, if $K$ is the total number of irreducible components, let
$l_1,l_2,\dots,l_K$ be the sizes of those components, arranged in nonincreasing order.
We suppose that any  multiple copies are always included so that, e.g., we have 
$l_1 + l_2 + \cdots + l_K =n$. Let $C_1,\dots,C_n$ be the numbers of components of our object, 
of sizes $1,\dots,n$ respectively;  then also $C_1 + 2C_2 + \cdots +nC_n =n$. Note that if indices
$i_1,\dots,i_k$ have a sum exceeding $n$, then necessarily at least one of the counts
$C_{i_1},\dots,C_{i_k}$ must vanish.

 We are given a sequence of probability distributions, one for each $n$, on the objects of
size $n$. Thus $K$, the sizes $l_1,l_2,\dots,l_K$ and the counts $C_1,\dots,C_n$ all become
random variables, for each value of $n$.

Also consider a sequence of families, one family for each $n$, of $k$-tuples $(i_1,\dots,i_k)$
of distinct indices $1 \le i_1,\dots,i_k \le n$, for fixed $k$, with all ratios
$i_1/n,\dots,i_k/n$ bounded away from $0$ as $n \to \infty$, uniformly over the whole 
sequence. Call such a sequence a {\em good sequence of families of $k$-tuples}.

 We can now state the Master Theorem:
\begin{theorem}\label{masterthm}
 Suppose our combinatorial objects and probability distributions are such that that 
for some $\theta >0$ and for any good sequence of families of $k$-tuples $(i_1,\dots,i_k)$,
the expected values $E\{C_{i_1}\cdots C_{i_k}\}$ satisfy
\begin{equation}\label{masterineq}
E\{C_{i_1}\cdots C_{i_k} \} =  
\frac{\theta^k}{(1-m/n)^{1- \theta}}\frac{1}{i_1 i_2\cdots i_k}\left(1+o(1)\right),
\end{equation}
uniformly over the sequence of families, where we write 
$m = m(i_1,\dots,i_k) =: i_1 + \cdots + i_k$.
Then, as $n \to \infty$ the joint distribution of the initial $k$-long sequence of scaled sizes
$$
l_1/n,\dots,l_k/n
$$
converges to the initial $k$-dimensional projection of PD($\theta$).
\end{theorem}
\begin{proof}
 We apply Proposition \ref{maintheta}. In the notation of that proposition, 
let $A_n$ be the multisubset containing the elements $l_1/n,l_2/n,\dots$, and 
let $I_j = [a_j,b_j]  \subset (0,1], j = 1,\dots,k,$ be disjoint intervals with 
$a_j>0$ for $j=1,\dots,k$ and with $b_1 + \cdots + b_k <1$. 
Then we have
$$
 E\{\: |A_n \cap [a_1,b_1]|\cdots  |A_n \cap [a_k,b_k]|\:\}
 = E\{ \sum_{a_1n < i_1 \le b_1n}C_{i_1} \cdots \sum_{a_kn < i_k \le b_kn}C_{i_k}\:\} 
$$

$$
=  \sum_{a_1n < i_1 \le b_1n} \cdots \sum_{a_kn < i_k \le b_kn} E\{C_{i_1} \cdots  C_{i_k}\}
$$
$$
=    \sum_{a_1n < i_1 \le b_1n} \cdots \sum_{a_kn < i_k \le b_kn} 
\frac{\theta^k}{(1-m/n)^{1- \theta}}\frac{1}{i_1 i_2\cdots i_k}\left(1+o(1)\right)
$$
where we may appeal to \eqref{masterineq} in the last step because we claim that the 
sequence of families of $k$-tuples of indices arising, as $n \to \infty$, from the 
$k$-fold summations, forms a good sequence: The indices in each $k$-tuple are distinct 
because the intervals
$I_1,\dots,I_k$ are disjoint, and the ratios $i_1/n,\dots,i_k/n$ are uniformly bounded away 
from $0$ because for each $j$ we have $ 0< a_j \le i_j/n$ where the numbers $a_1,\dots,a_k$ 
are independent of $n$. 

 Note that always, $a_1 + \cdots + a_k \le m/n \le b_1 + \cdots + b_k$.  If $\theta \le 1$, 
we may then write 
$$
E\{\: |A_n \cap [a_1,b_1]|\cdots  |A_n \cap [a_k,b_k]|\:\}
$$
$$
\ge \frac{\theta^k}{(1-a_1 - \cdots - a_k)^{1- \theta}}
\left(\sum_{a_1n < i_1 \le b_1n} \frac{1}{i_1}\right)\cdots \left(\sum_{a_kn < i_k \le b_kn} \frac{1}{i_k}\right)
\left(1+o(1)\right)
$$ 
$$
= \frac{\theta^k}{(1-a_1 - \cdots - a_k)^{1- \theta}}\prod_{j=1}^k(\log b_j -\log a_j) + o(1).
$$
If $\theta >1$ we proceed in the same way, except that $-a_1-\cdots-a_k$ is 
replaced with $-b_1-\cdots-b_k$.
In either case, take $\liminf_{n \to \infty}$ of both ends of the inequality and apply 
Proposition \ref{maintheta}, with $\alpha = 1-\theta$ and $\beta =0$ or with 
$\beta = 1-\theta$ and $\alpha =0$, respectively. 
\end{proof}

 Of course, in any application of Theorem \ref{masterthm}, the burden will be the establishment
of \eqref{masterineq} with the required uniformity guarantees.  Conveniently, well-honed tools
for this already exist.

\subsection{Exp-log asymptotics}
 We will use two formulas of Flajolet and Soria \cite{FSo}, \eqref{comp} and 
\eqref{phiobj} below,
which we extract from the exposition in~\cite[Section~VII.2]{FSe}, together with several others 
in the same spirit. They are conveniently packaged consequences of asymptotic formulas of 
Flajolet and Odlyzko, especially designed for certain
combinatorial applications.\footnote{Formulas \eqref{comp} and \eqref{phiobj} originally 
appeared as preliminary
results in \cite{FSo}, where they were used as ingredients for various other limit theorems.}
Let G(z) be a function of a complex variable 
analytic near $z=0$, whose series expansion at $0$ has real non-negative coefficients
with finite radius of convergence $\rho$. We assume, with Flajolet and Soria, 
that for some $\theta>0$ and some real 
$\lambda$ 
\begin{description}
\item[FS 1] $\rho$ is the unique singularity of $G(z)$ on $|z| = \rho$;
\item[FS 2] G(z) is continuable to a slightly larger open domain $\Delta$ 
consisting of a disc of radius exceeding $\rho$ centered at $0$, but possibly excluding a 
closed acute-angled wedge domain $|\arg(z-\rho)| \le \gamma$ with vertex at $\rho$, for some 
$0 \le \gamma < \pi/2$;
\item[FS 3] we have
$$
G(z) = \theta \log \frac{1}{1-z/\rho} + \lambda +O\left(\frac{1}{(\log(1-z/\rho))^2}\right)
$$
as $z\to \rho$ in $\Delta$.
\end{description}
For later reference, note that 
$$
\log \frac{1}{1-z/\rho}
$$
itself certainly satisfies all three items, continuing analytically, as it does, 
to the complement of the real ray $z \ge \rho$.

 Given such a $G$, let $\phi >0$. Then the formulas of Flajolet and Soria are as follow:
the power series coefficients of $G$ around $z=0$ satisfy
\begin{equation}\label{comp}
[z^n]G(z) = \frac{\theta}{n}\rho^{-n}\left(1 + O\left((\log n)^{-2}\right)\right),
\end{equation}
%
and if $F(z) = \exp(\phi G(z))$  then 
\begin{equation}\label{phiobj}
[z^n]F(z) = \frac{e^{\phi \lambda}}{\Gamma(\phi \theta)}n^{\phi \theta-1} \rho^{-n}
\left(1 + O\left((\log n)^{-2}\right)\right).
\end{equation}

 Now add the restriction that 
$$
\rho < 1
$$
and assume that the numbers $g_i$ are {\em integers}.
Then if $F(z) = \exp(\phi G(z)+ R(z))$, where 
\begin{equation}\label{selR}
R(z):= \sum_{j \ge 2}(-1)^{j+1}\phi^j G(z^j)/j
\end{equation}
then
\begin{equation}\label{selobj}
[z^n]F(z) = \frac{Ce^{\phi \lambda }}{\Gamma(\phi \theta)}n^{\phi \theta-1} \rho^{-n}
\left(1 + O\left((\log n)^{-2}\right)\right)
\end{equation}
for a certain nonzero constant $C$ to be described.

 Finally, also restrict $\phi$ to 
$$
\rho^{-1} > \phi >0
$$
but release the restriction of the numbers $g_i$ to integers.
Then if $F(z) = \exp(\phi G(z)+ R(z))$, where this time
\begin{equation}\label{mulR}
R(z):= \sum_{j \ge 2}\phi^j G(z^j)/j,
\end{equation}
then we get
\begin{equation}\label{mulobj}
[z^n]F(z) = \frac{Ce^{\phi \lambda}}{\Gamma(\phi \theta)}n^{\phi \theta-1} \rho^{-n}
\left(1 + O\left((\log n)^{-2}\right)\right)
\end{equation}
with $C \ne 0$, same as \eqref{selobj}, once again.

 As mentioned, \eqref{comp} and \eqref{phiobj} are proved
in~\cite[Section~VII.2]{FSe}.\footnote{Formula \eqref{phiobj} is
  actually proved there for $\phi = 1$; but the more general formula
  is a trivial corollary of that one.} 
 
 As for \eqref{mulobj}, we claim that $R(z)$ as defined in \eqref{mulR} 
is analytic in an open disc about $0$ of some radius exceeding $\rho$. If so, then since 
$$
R(z) - R(\rho) = O(z-\rho) = O\left(\frac{1}{(\log(1-z/\rho))^2}\right)
$$ near $z = \rho$, we find that \eqref{mulobj} is a corollary of 
\eqref{phiobj}, with  $C= \exp(R(\rho))$,
if $\phi G(z) + R(\rho)/\phi + (R(z) - R(\rho))/\phi$ replaces $G$ in {\bf FS 3}. 

 To see that $R(z)$ is as claimed, note that for $j\ge 2$ 
each function $G(z^j)$ is analytic in the open disc of radius $\rho^{1/j} \ge \rho^{1/2} >\rho$, 
and also that they are uniformly $O(z^2)$ in any closed disc of radius less than  $\rho^{1/2}$.
Also, when $\phi < \rho^{-1}$, we have $|\phi z|<1$ in the open disc of radius 
$\min\{\phi^{-1},\rho^{1/2}\}$; and this radius exceeds $\rho$.  Therefore,
the series defining $R(z)$ converges uniformly and absolutely in any compact subset of that disc.
(This argument too was given by Flajolet and Soria, for $\phi =1$.) This proves \eqref{mulobj}.

 Although the same argument also works for \eqref{selobj}, given \eqref{selR}, 
provided $\phi < \rho^{-1}$, for larger $\phi$ it gets the series \eqref{selR} 
defining $R(z)$ to converge only for $|z| < \phi^{-1} \le \rho$, which is not 
good enough. To derive \eqref{selobj} for all positive $\phi$ we need to look 
under the hood a bit.  

 From {\bf FS 3} we have 
$$
F(z) = \exp(G(z)) = 
e^{\lambda} (1-z/\rho)^{-\theta}\left(1 + O\left(\frac{1}{(\log(1-z/\rho))^2}\right)\right),
$$
and it is {\em this} formula from which \eqref{phiobj} follows 
via results of Flajolet and Odlyzko; see the
discussion in~\cite[Section~VII.2 ]{FSe}. From $F(z) = \exp(\phi G(z)+ R(z))$, then, we get
\begin{equation}\label{expression}
F(z) =  
e^{R(z)}e^{\phi \lambda} (1-z/\rho)^{-\phi \theta}\left(1 + O\left(\frac{1}{(\log(1-z/\rho))^2}\right)\right).
\end{equation}

 Regardless of the behavior of $R(z)$, we claim that 
\begin{lemma}\label{Sanal}
 For any fixed $\phi >0$,
$S(z):= e^{R(z)}$ continues analytically to a disc around $0$ of radius greater than $\rho$,
and $S(\rho) \ne 0$.
\end{lemma} 
 If so, then  from replacing $e^{R(z)}$ in \eqref{expression} with 
$S(z) = S(\rho) \times \frac{S(z)}{S(\rho)} = S(\rho)\left(1 + O(z-\rho)\right)$ near $z=\rho$,
we immediately deduce \eqref{selobj}, with $C= S(\rho)$.  So it remains to prove the lemma,
which we now do.

 Fix $\phi > 0$. If we restrict to the domain
$$
D = \{z: |z| < \min(\phi^{-1},\rho^{1/2})\},
$$
then from rearranging the Taylor expansions of $\log$ terms we get 
\begin{equation}\label{Rseries}
R(z) = \sum_{i \ge 1}g_i\left(\log(1+\phi z^i) - \phi z^i\right),
\end{equation}
a valid identity between analytic functions on $D$.

 Now pick an index $\xi >0$ for which $\phi \rho^{\xi/2} <1$, and set
$$
T(z): = \sum_{i \ge \xi}g_i\left(\log(1+\phi z^i) - \phi z^i\right).
$$
Since on any compact subset of $\{|z| < \rho^{1/2}\}$ the terms 
$g_i\left(\log(1+\phi z^i) - \phi z^i\right)$ are $O(\phi^2g_i z^{2i})$, uniformly for $i \ge \xi$,
we see from $\eqref{comp}$ that $T(z)$ defines an analytic function on the open disc 
$\{|z| < \rho^{1/2}\}$.  Also, since we have assumed that the $g_i$'s are non-negative integers, the
expressions $(1+\phi z^i)^{g_i}$ are polynomials, hence certainly 
single valued and analytic on the same disc. Therefore the formula
$$
S(z) = e^{R(z)} = 
\left(\prod_{1\le i < \xi} \left((1+\phi z^i)\exp(-\phi z^i)\right)^{g_i}\right) e^{T(z)}
$$ 
continues $S(z)$ analytically to the open disc $\{|z| < \rho^{1/2}\}$; and by 
inspection\footnote{Note that since $S(z)$ does possess zeroes for $|z| < \rho^{1/2}$ when $\phi$ is
large enough, $R(z)$ itself {\em cannot} then continue to that domain.}
we have $S(\rho) \ne 0$. This completes the proof of the lemma and, hence, of \eqref{selobj}.

\section{The three combinatorial families} 
\subsection{Assemblies}

 A permutation of length $n$ may be thought of as a partition of $[n] := \{1,\dots,n,\}$
into disjoint nonempty blocks, where on each block of size $i$ one of $m_i = (i-1)!$ possible cycle 
structures is imposed.  More generally, given a sequence $m_1,m_2,\dots$ of positive integers
an {\em assembly of size $n$} is a partition of $[n]$ into disjoint nonempty blocks, 
where on each block of size $i$ one of $m_i$ possible structures is imposed, 
called ``irreducible''.\footnote{In examples
of interest the numbers $m_i$ are not arbitrary, of course -- they are the numbers of irreducible
combinatorial objects of some sort, of sizes $i$.} If 
$$
M(x) = \sum_{i \ge 1} m_i x^i /i!
$$
and
$$
Q(x) = \sum_{n \ge 0} q(n)x^n/n!
$$ 
are the exponential generating functions for the numbers of irreducible objects of sizes $i$
and the total numbers of assemblies on the set $[n]$,
then it is well-known that assemblies are characterized by the formula
\begin{equation}
Q(x) = \exp(M(x)).
\end{equation}
(Conventionally, we have $q(0) = 1$.)
Further, if $q(n,k)$ is the number of objects of size $n$ and with $k$ irreducible components, 
then if we write
$$
q_{\phi}(n) = \sum_{k=1}^n q(n,k)\phi^k
$$
and
$$ 
Q(x,\phi) = \sum_{n \ge 0} q_{\phi}(n)x^n/n!
$$
for some positive parameter $\phi$, we have 
\begin{equation}
Q(x,\phi) = \exp(\phi M(x)).
\end{equation}
See, e.g.,~\cite[Section 9.1]{AT94}.

 Given a family of assemblies, i.e. given the sequence $m_1,m_2,\dots$, suppose an
assembly of size $n$ is picked at random, either uniformly or, more generally, from the
tilted distribution with parameter $\phi$. Let $C_1,\dots,C_n$ be the counts of its irreducible
components of sizes $1$ through $n$, respectively. 
For any $k$-tuple of distinct positive indices $i_1,\dots,i_k$  with $m = i_1+ \cdots + i_k \le n$,
the following expression for the mixed moment  $E\{C_{i_1}\cdots C_{i_k} \}$
is a special case of formula (126) of \cite{AT94}, specialized down to simple products:
\begin{equation}\label{assembmom}
E\{C_{i_1}\cdots C_{i_k} \} = 
\rho^{-m}\frac{n!}{q_{\phi}(n)}\frac{q_{\phi}(n-m)}{(n-m)!}\prod_{j=1}^{k}\left(\frac{\phi m_{i_j}\rho^{i_j}}
{i_j!}\right).
\end{equation}

 We can combine \eqref{assembmom} with the Flajolet-Soria formulas discussed above. 
We suppose we are given a family of assemblies with exponential generating function
$M(x) = \sum_{i \ge 1} m_i x^i /i!$ for the numbers of irreducible objects of sizes $1,2,\dots$.
\begin{lemma}\label{assemblemma}
 If conditions \textbf{FS 1, FS 2,} and \textbf{FS 3} are satisfied for $G(z) = M(z)$,
then for arbitrary $\phi > 0$ we have
\begin{equation}\label{assembineq}
E\{C_{i_1}\cdots C_{i_k} \} =  
\frac{(\phi \theta)^k}{(1-m/n)^{1-\phi \theta}}\frac{1}{i_1 i_2\cdots i_k}
\left(\prod_{j=1}^k\left(1 + O\left(\frac{1}{(\log i_j)^2}\right)\right) +o(1)\right).
\end{equation}
\end{lemma}
\begin{proof} 
With $F(z) = \exp(\phi M(z)) = Q(z,\phi)$, plugging \eqref{comp} and \eqref{phiobj} 
into \eqref{assembmom} immediately yields \eqref{assembineq}, uniformly over all
$k$-tuples of distinct positive indices $i_1,\dots,i_k$  with $i_1+ \cdots + i_k \le n$.
\end{proof}

 We can now give the main result:
\begin{theorem}
Let $l_1 \ge l_2 \ge \cdots $ be the irreducible component sizes of a random assembly on the
set $[n]$, chosen
from a tilted distribution with parameter $\phi$, and define $L_{jn} = l_j/n$, where the latter
sequence is padded out with zeros.  Suppose the Flajolet-Soria conditions 
\textbf{FS 1, FS 2,} and \textbf{FS 3} are satisfied when $G(z) = M(z)$. Then for each $k>0$,
the joint distribution of $L_{1n},\dots,L_{kn}$ converges to the initial $k$-dimensional
joint PD($\phi \theta$) distribution.
\end{theorem}
\begin{proof}
Formula \eqref{assembineq} in Lemma \ref{assemblemma} looks ready to serve as 
formula \eqref{masterineq} in Theorem \ref{masterthm}, except for the $k$-fold product towards
the end of 
\eqref{assembineq}. However we are allowed to restrict attention, when applying that theorem, 
to good families of $k$-tuples of indices, i.e. with a 
uniform positive lower bound hypothesis 
on the ratios $i_1/n,\dots,i_k/n$. This converts the $k$-fold product to $(1 + o(1))$. 
Now apply the theorem.
\end{proof}   

\subsection{Multisets and Selections}
 Multisets and Selections are sufficiently alike that we can treat them simultaneously,
in parallel.

 A monic polynomial of degree $n$, over some finite field, may be unambiguously identified 
with the multiset consisting of its irreducible monic factors -- ``multiset'', because 
some factors could appear with multiplicities; 
and the degrees add up to $n$.  Or, if we are interested only in squarefree polynomials, then 
the irreducible factors form a set, with no repetition of elements; and the degrees
still add up to $n$. More generally,
suppose we are given some universe of ``irreducible'' objects having positive integer weights,
with exactly $m_i$  different kinds of irreducibles of weight $i$. Our two polynomial
examples are prototypical of the following two respective constructions. 
\begin{itemize}
\item A {\em combinatorial multiset of weight $n$} 
is a multisubset of our universe, with total weight $n$. Equivalently, the 
{\em integer} $n$ is partitioned into positive summands, and for each summand $i$ one of the $m_i$
possible summands of weight $i$ is selected, with replacement.
\item A {\em combinatorial selection of weight $n$} is a subset of our universe,
of total weight $n$. So all components of an object must be of distinct kind, 
though distinct components of the same weights are permitted.
\end{itemize}
So the selection construction could be viewed as a subclass of the multiset construction.  
Note that any additional structure associated with our universe, 
for instance the fact that a collection
of irreducible polynomials multiplies together to form another polynomial, 
need not be considered in discussion of counting formulas.

 For either construction, let 
$$
M(x) = \sum_{i \ge 1} m_i x^i
$$
be the ordinary generating function for the numbers of irreducibles of weight $i$.
Also, if $q(n,k)$ denotes the number of multisets of total weight $n$ containing 
$k$ irreducibles, 
including multiplicities, or if we let it denote the number of selections of 
total weight $n$ containing $k$ irreducibles,
then in either case, for any given positive $\phi$ write 
$$
q_{\phi}(n) = \sum_{k = 1} ^n q(n,k)\phi^k
$$  
and
$$
Q(x,\phi) = \sum_{n \ge 0} q_{\phi}(n)x^n.
$$
For $\phi =1$, in either case, the series $Q = Q(x,1)$ reduces to 
the ordinary generating function for the numbers of composite objects of weights $n$.

 The following two formulas connecting $Q(x,\phi)$ and $M(x)$ are well-known:
For multisets we have
\begin{equation}\label{multigen}
Q(x,\phi) = \prod_{i \ge 1}(1-\phi x^i)^{-m_i} = \exp \left(\sum_{j \ge 1} \phi^j M(x^j)/j \right), 
\end{equation}
and for selections we have
\begin{equation}\label{selectgen}
Q(x,\phi) = \prod_{i \ge 1}(1+\phi x^i)^{m_i} = \exp \left(\sum_{j \ge 1} (-1)^{j+1}\phi^j M(x^j)/j \right). 
\end{equation}
See, e.g.,~\cite[Section~9.2]{AT94} for~\ref{multigen} and Section~9.3 for \ref{selectgen}. 

 In either construction, given the set of all composite objects of
total weight $n$ constructed from some given universe of irreducibles, suppose one object 
is picked at random according to the tilted distribution 
with tilting parameter $\phi$. Let $C_1,\dots,C_n$ be the
numbers of irreducible components of that object, of weights $1$ through $n$ respectively, 
including any multiple occurrences.
(So again  $C_1 + 2C_2 + \cdots + nC_n = n$.)

 For any sequence of positive indices $i_1, \dots,i_k$ with $i_1 + \cdots + i_k \le n$ 
we have a formula for the corresponding mixed moment.
For the multiset construction it is
\begin{equation}\label{multimom}
E\{C_{i_1}\cdots C_{i_k}\} = 
\frac{m_{i_1}\cdots m_{i_k}}{q_{\phi}(n)}\sum_{h_1,\dots,h_k \ge 1} \phi^{h_1 + \cdots + h_k} q_{\phi}(n-h_1i_1 - \cdots - h_ki_k),
\end{equation}
and for the selection construction it is
\begin{equation}\label{selectmom}
E\{C_{i_1}\cdots C_{i_k}\} = 
\frac{m_{i_1}\cdots m_{i_k}}{q_{\phi}(n)}\sum_{h_1,\dots,h_k \ge 1} (-1)^{h_1 + \cdots + h_k+k}\phi^{h_1 + \cdots + h_k} q_{\phi}(n-h_1i_1 - \cdots - h_ki_k).
\end{equation}
(See formulas (139) and (146) in~\cite[Sections~9.2 and~9.3]{AT94}, respectively.  
While the authors give explicit formulas only for the individual
falling factorial moments, their method of proof easily yields the present 
formulas as well.) 
Note that because of the expressions $q_{\phi}(n-h_1i_1 - \cdots - h_ki_k)$, the sums have finitely
many terms. In our application to Theorem \ref{mulselPD} below only the leading term in each 
case, where $h_1 = \cdots = h_k =1$, will matter asymptotically.

 We can marry the Flajolet-Soria asymptotics to the moment formulas \eqref{multimom} and 
\eqref{selectmom}:
\begin{lemma}\label{multlemma}  
Suppose that conditions \textbf{FS 1, FS 2,} and \textbf{FS 3} are satisfied for $G(z) = M(z)$,
the ordinary generating function of the $m_i$'s, and that for the multiset construction we impose
$\phi < \rho^{-1}$. For the selection construction we allow $\phi$ to be arbitrarily large. 
Then in either case we have
we have
\begin{equation}\label{multineq}
E\{C_{i_1}\cdots C_{i_k} \} =  
\frac{(\phi \theta)^k}{(1-m/n)^{1-\phi \theta}}\frac{1}{i_1 i_2\cdots i_k}
\left(1+O(\lfloor n/i_1 \rfloor\cdots\lfloor n/i_k \rfloor \rho^{\min(i_1,\dots,i_k)} )\right),
\end{equation}
where $m = i_1 + \cdots + i_k $.
\end{lemma}
\begin{proof}
Note that in the present cases the radius of convergence $\rho$ of $G(z) = M(z)$
must satisfy $\rho < 1$, for the trivial reason that otherwise the coefficients of $G(z)$ 
as given in \eqref{comp} could not yield integers, for large enough $n$.

 That being so, substitute \eqref{comp} and either \eqref{mulobj} or \eqref{selobj} into 
\eqref{multimom} or \eqref{selectmom} respectively.  It is then straightforward to get 
\eqref{multineq}.
\end{proof}

 We can now give the Poisson--Dirichlet limit theorem for multisets and selections.  
We suppose we are given a
universe of irreducibles with ordinary generating function $M(x)$ for the numbers $m_i$ of different
kinds of weight $i$, for $i\ge 1$.  Let $\rho <1$ be the radius of convergence of $M$.
\begin{theorem}\label{mulselPD}
Let $l_1 \ge l_2 \ge \cdots $ be the irreducible component sizes of a 
random multiset or a random selection of weight $n$, 
chosen from a tilted distribution with parameter $\phi$, where for multisets we suppose that 
$\phi < 1/\rho$. Define $L_{jn} = l_j/n$, where the latter
sequence is padded out with zeros.  Suppose the Flajolet-Soria conditions 
\textbf{FS 1, FS 2,} and \textbf{FS 3} are satisfied when $G(z) = M(z)$. Then for each $k>0$,
the joint distribution of $L_{1n},\dots,L_{kn}$ converges to the initial $k$-dimensional
joint PD($\phi \theta$) distribution.
\end{theorem}
\begin{proof}
We appeal once again to Theorem \ref{masterthm}. Restricting to good sequences of families of
$k$-tuples,
together with the fact that $\rho <1$,  converts the
multiplicative error term in \eqref{multineq} to $1 + o(1)$, as required in \eqref{masterineq}.
So the theorem applies.
\end{proof}

\noindent\textit{Remark}. The necessity for the restriction to $\phi < \rho^{-1}$ for multisets 
may appear to be an artifact
of our complex analytic methodology, but the same restriction is also imposed with the 
methodology of
\cite{ABT}.  In fact, as far as we are aware, the limiting behavior for $\phi\ge \rho^{-1}$
is unknown, for multisets.

\bibliography{P-D}
\bibliographystyle{plain}

\end{document}